\documentclass[11pt,reqno]{amsart}

\usepackage{amscd}
\usepackage{amsfonts}
\usepackage{amsmath}
\usepackage{amssymb}
\usepackage{amsthm}
\usepackage{fancyhdr}
\usepackage{latexsym}
\usepackage[colorlinks=true, pdfstartview=FitV, linkcolor=blue,
            citecolor=blue, urlcolor=blue]{hyperref}

\usepackage{cancel}

\input epsf
\input texdraw
\input txdtools.tex
\input xy
\xyoption{all}

\usepackage{caption}
\usepackage{subcaption}

\usepackage{cases}
\usepackage{cleveref}

\usepackage{xurl}

\usepackage{color}

\definecolor{Red}{rgb}{1.00, 0.00, 0.00}
\definecolor{DarkGreen}{rgb}{0.00, 1.00, 0.00}
\definecolor{Blue}{rgb}{0.00, 0.00, 1.00}
\definecolor{Cyan}{rgb}{0.00, 1.00, 1.00}
\definecolor{Magenta}{rgb}{1.00, 0.00, 1.00}
\definecolor{DeepSkyBlue}{rgb}{0.00, 0.75, 1.00}
\definecolor{DarkGreen}{rgb}{0.00, 0.39, 0.00}
\definecolor{SpringGreen}{rgb}{0.00, 1.00, 0.50}
\definecolor{DarkOrange}{rgb}{1.00, 0.55, 0.00}
\definecolor{OrangeRed}{rgb}{1.00, 0.27, 0.00}
\definecolor{DeepPink}{rgb}{1.00, 0.08, 0.57}
\definecolor{DarkViolet}{rgb}{0.58, 0.00, 0.82}
\definecolor{SaddleBrown}{rgb}{0.54, 0.27, 0.07}
\definecolor{Black}{rgb}{0.00, 0.00, 0.00}
\definecolor{dark-magenta}{rgb}{.5,0,.5}
\definecolor{myblack}{rgb}{0,0,0}
\definecolor{darkgray}{gray}{0.5}
\definecolor{lightgray}{gray}{0.75}

\usepackage[pdftex]{graphicx}
\usepackage{epstopdf}

\newcommand{\rr}{\mathbb{R}}
\newcommand{\p}{\partial}

\theoremstyle{plain}  
\newtheorem{theorem}{Theorem}
\newtheorem{proposition}{Proposition}
\newtheorem{lemma}{Lemma}

\theoremstyle{definition}

\usepackage{tikz}

\usetikzlibrary{patterns}

\begin{document}

\title{
On the 
well-posedness of a nonlinear diffusive  \\ SIR epidemic model
}
  
\author{ Curtis Holliman and Harry Prieto}


\begin{abstract}
This work considers an extension of the SIR equations from epidemiology that includes a spatial variable.  This model, referred to as the Kermack-McKendrick equations (KM), is a pair of 
diffusive partial differential equations,   and methods developed for the Navier-Stokes equations and models of fluid dynamics are adapted to prove that KM is well-posed in the homogenous
Sobolev spaces with exponent $0 \le s <2$.
\end{abstract}

\date{September 18, 2021}

\maketitle

\markboth{On the 
well-posedness of a nonlinear diffusive  SIR epidemic model}
{C. Holliman and H. Prieto}

\section{Introduction}
\label{sec:1}
 \setcounter{equation}{0}

The outbreak of the COVID-19 pandemic galvanized the efforts to improve the predictive power of the mathematics modeling the spread of disease.  The most well-known of these models is given by the SIR equations. The present paper considers an extension of these equations that includes a
spatial variable. These  equations, which we call the Kermack-McKendrick equations (KM), 
change the SIR model from a set of ordinary differential equations (ODEs) into a coupled pair of
diffusive partial differential equations (PDEs).  We investigate  the well-posedness of
KM  using methods that were originally developed for fluid dynamics, in particular
for the Navier-Stokes equations (NS).

The SIR model was pioneered by W. Kermack and A. McKendrick in \cite{km} and is an example of
a  compartmental model.  In  their 
original formulation, the population 
is partitioned into the disjoint groups, or compartments, consisting of  the {\bf susceptible} $(S)$, 
{\bf infectious} $(I)$,    and 
{\bf recovered} $(R)$ individuals.  These quantities are strictly functions of time $t$, and the ODEs they give rise to, called the SIR-equations,  are given by
\begin{align}
\notag
S'(t)  &= -\beta SI, \\
I'(t) &= \beta SI -\mu I,\\ \notag
R'(t) &= \mu I. 
\end{align}
Here $\beta$ and $\mu$ are constants representing the transmission and recovery rates, respectively. For a detailed description of the basic assumptions of the model and the technical underpinnings that lead to its equations we refer to  \cite{br}, \cite{jo},  \cite{li}, \cite{mbh}.
For a review of the history of the SIR equations, the interested reader may consult \cite{bc}

Since its inception, most of the attention has concentrated on using the SIR model to understand disease transmission, and over the years important applications to public health have been found \cite{we}. A prime example of this is vaccination, where the transition rate between compartments is accelerated, since vaccinated individuals can be immediately placed in the $R$, (recovered) compartment. Kermack and McKendrick applied their model to the 1906 bubonic outbreak in Bombay \cite{bc}, but the model has also been employed in 
a wide variety of circumstances such as the evolution of the dengue outbreaks in Cuba (1997) and Venezuela (2000) \cite{ho},  the classical swine flu in the Netherlands  (1997-1998) \cite{ar}, and many others.

Considerable work has been devoted to improving the SIR model itself.
Notably, the original model has been expanded by the addition of more compartments. For instance, along with the traditional three compartments, some models also include incubation and latency periods ($E$) \cite{li}. More recently, some models have incorporated compartments to account for immunization and vaccination in populations \cite{fu}, \cite{xu}.

The descriptive and predictive power of the model has been  applied to the COVID-19 pandemic. The great interest generated by the topic and its timely nature are evidenced by the explosion of the literature on the subject. For a few applications of the SIR model in this context, we refer the readers to \cite{al}, \cite{co}, \cite{gou}.

One fundamental issue in using the SIR equations to  model a pandemic, however, is that it completely 
ignores spatial information.  As compared to a localized disease outbreak, the location  and concentration of  affected individuals in a global setting would most  certainly contribute to the time evolution of the model.
With this issue in mind, a generalized compartmental SIR model is constructed
by allowing  individuals to move via random walks.  For an investigation of random walks in 
this context, we refer the reader to \cite{s}, and for the foundational work on 
Brownian motion, which lies at the core of these diffusion processes, we refer the reader to \cite{eins}.
According to these models, an individual  moves randomly in a direction with the amplitude of the Brownian 
motion equalling  $D_S$ and $D_I$ for the Susceptible and Infected individuals, respectively.  After taking an expected value, these constants become the coefficients in the diffusion linear symbol, 
and we obtain the equations
\begin{equation}
\label{km-original}
\begin{split}
&S_t = D_S  \Delta S -\beta SI,\\
&I_t = D_I \Delta I + \beta SI -\mu I.
\end{split}
\end{equation}
It is worth noting that if there is no displacement of the individuals, then $D_S = D_I = 0$, and we 
obtain the original SIR equations.
A detailed description of how to obtain \eqref{km-original}  appears in \cite{cd}, and is reviewed in \cite{du}, \cite{f}.  Additionally, the traveling wave solutions of \eqref{km-original} have
been  investigated in  \cite{bz}, \cite{hi}, \cite{wz}, \cite{ww}.  

In this work, we will assume the spatial dimension in \eqref{km-original} to be one.
 Additionally, we make the change of variables 
$u(x,t) = S( x \sqrt{D_S/\beta}   ,t/\beta)$
and $v(x,t) = I( x \sqrt{D_I /\beta }, t/\beta)$ to obtain a slightly simpler non-dimensional version of \eqref{km-original}.  In this new version, the coefficient for the linear $v$ term will be $-\mu \beta$ which we 
will relabel as $-\mu$.  We refer to this new system of equations as the Kermack-McKendrick equations  (KM), and they are given by
\begin{align}
\notag
& u_t +uv - u_{xx}    = 0,  \\ \label{km}
&v_t - uv-v_{xx} - \mu  v  = 0,  \\ \notag
&u(x,0) = \varphi(x),  \quad v(x,0) = \psi (x), \quad t \in [0,T), x \in \rr.
\end{align}
We will approach the initial value  problem (ivp) 
posed by \eqref{km} using theory developed in the study of diffusion equations in fluid dynamics. To that end, we adapt the methods developed and applied in   \cite{be}, \cite{ho},  \cite{kf}.
Specifically, we investigate the {\it well-posedness} of KM in the sense of Hadamard. 
In order to  rigorously state what we mean by well-posedness, we must also
state precisely what spaces we are taking the initial data and solutions to be in.
We will take the initial data $\varphi, \psi$ to be in the homogeneous Sobolev spaces $\dot{H}^s$, and 
the solution $u$ to be in the intersection of $\dot{H}^s$ and the time-weighted $L^4$ spaces,
which we will call $X^s$.  The precise definitions of these spaces are provided  for the reader in 
\eqref{hs-def} and \eqref{xs-def}, respectively.

The idea of well-posedness was introduced  in \cite{ha}, 
and 
we say 
that the KM equations are  well-posed with initial data in 
$(\varphi, \psi) \in \dot{H}^s \times \dot{H}^s$
and solution in $ (u,v) \doteq W(\varphi, \psi) \in X^s \times X^s$,
 if the following three conditions hold:

\begin{itemize}
\item[I)] {\bf Existence.} For any initial data 
$(\varphi, \psi) \in \dot{H}^s \times \dot{H}^s$,
there exists a solution $(u,v) \in X^s \times X^s$ to KM.
\item[II)] {\bf Uniqueness.} The solution $(u,v) \doteq W(\varphi, \psi)$ is unique in  the space $X^s \times X^s$.
\item[III)]  {\bf Continuity/Stability.} The solution map 
$W : \dot{H}^s \times \dot{H}^s \to X^s \times X^s$ is continuous. 
\end{itemize}
With this definition in mind, we now state the primary result of this work.  



\begin{theorem}
\label{main}
Let $ 0 \le s <  2$, $0 < T   < \frac{1}{6\mu}$ and 
 $(\varphi, \psi) \in \dot{H}^s \times \dot{H}^s$
satisfying the smallness condition
\begin{align}
\| (\varphi, \psi) \|_{\dot{H}^s \times \dot{H}^s} \le \frac{1}{18 C_\ell C_b},
\end{align}
where the constants $C_\ell$ and $C_b$ are given in 
Propositions  \ref{linear-est}  and \ref{bilinear}, respectively.  Then 
the KM ivp \eqref{km} has a unique solution $(u,v) \in X^s\times X^s$.
Moreover, the solution map, 
$W: \dot{H}^{s} \times \dot{H}^s \to X^{s} \times X^s $, which takes $ (\varphi, \psi) 
\mapsto  (u,v)$, is Lipschitz
continuous. 
\end{theorem}
The proof of Theorem \ref{main} revolves around the techniques developed in \cite{kf} to prove the well-posedness of
the Navier-Stokes  (NS) equations
\begin{equation}
\begin{split}
& u_t + (u \cdot \nabla) u + \mu \Delta u + \nabla p = 0, \\
&\text{div} \; u = 0,
\end{split}
\end{equation}
where $p$ is the pressure of the fluid, and $\mu$ its viscosity.   Here, the  
Brownian motion amplitudes  in \eqref{km-original}, 
$D_S$ and $D_I$,
act in a similar  manner as the viscosity  coefficient $\mu$ in NS. 
The strategy that was implemented for NS was built on the foundations developed in \cite{g2} and
consisted in showing that
the 
associated integral operator  had a fixed point in a suitable space.
These ideas have also been 
used in other hydrodynamic equations as such  the viscous Burgers (vB) equation  
\begin{align}
&u_t + u u_x + \mu u_{xx} = 0,
\end{align}
which was examined in \cite{be}, and the $k$-Burgers equation
\begin{align}
u_t + u^{k} u_x + \mu u_{xx} = 0, \quad\quad k = {1, 2,  3, \cdots}.
\end{align}
which was investigated in \cite{ho}.

\textbf{Outline of the paper.}
The present paper is organized as follows.  In section 2, we provide a number of preliminaries, including the definitions of our function spaces as well as linear estimates for the diffusion operator.  
In section 3, we first reformulate KM as a fixed point problem and then prove that this associated
integral operator is a contraction mapping.  In section 4, we provide a proof of the bilinear estimate
that was needed in order to establish the contraction in section 3.

\section{Preliminaries and Linear Estimates}
\label{sec:1}
 \setcounter{equation}{0}

In this section, we set up our notation and collect the basic estimates that will be used in the course of proving the main result.

\textbf{Notation.}
We say $A \lesssim B$ if there exists a constant $C >0$ such that $A \le CB$.
If $A \lesssim B$ and $B \lesssim A$ we write $A \simeq B$.

\textbf{Function Spaces.}
The spaces that we will use are the combination of 
homogeneous Sobolev spaces and time-weighted $L^p$ spaces, and we briefly provide a definition of these spaces and their norms.

We begin with the homogenous Sobolev space,  $\dot{H}^s$, which is a subspace of the Tempered Distributions where the following norm is finite. 
We 
take the 
the Riesz potential 
$D_x = (-\partial^2_x)^{1/2}$,
or equivalently,  the Fourier multiplier given by $\widehat{D_{x}f} (\xi) 
=  | \xi | \widehat{f} (\xi)$, and then define the $\dot{H}^s$-norm as 
\begin{align}
\label{hs-def}
\| u \|_{\dot{H}^s}
\doteq
\| D_x^s u \|_{L^2}
=
\Big(
\int_\rr |\xi|^{2s} | \widehat{u}(\xi)|^2 d\xi 
\Big)^{1/2}.
\end{align}

Next, we define our time-weighted $L^p$ spaces.
For any fixed $T$, $\alpha \ge 0$, we define 
the subspace $C^{\alpha}((0,T)$;$L^p) \subset C((0,T)$;$L^p)$ by
\begin{align*}
    C^\alpha_0 ((0,T);L^p )= 
    \Big\{ u\in C((0,T);L^p): \sup_{t \in (0,T)} t^\alpha \|u\|_{L^p} < \infty 
    \;\;  \text{and} \;\;  \lim_{t\to 0^+} t^\alpha \|u\|_{L^p} = 0
    \Big\}.
\end{align*}
For given $s,\alpha,p$ we can now define 
$X^{s,\alpha,p} = \dot{H}^s \cap     C^\alpha_0 ((0,T);L^p )$; however, 
in our particular  case, there is a relationship between $s$ and $\alpha$ that arises
in our proof of the Bilinear Estimate needed for Theorem \ref{main}.  We thus will restrict our attention to
\begin{align*}
\alpha \doteq \frac{1}{2} - \frac{s}{4}.
\end{align*}
Additionally, the only $p$ that we will utilize is $p = 4$ as it also arises
in the  Bilinear Estimate after applying  the generalized H\"{o}lder's inequality in $L^2$.  
In view of these choices, we define $X^s \doteq X^{s,\frac{1}{2} - \frac{s}{4}, 4}$, 
and take the norm to be
\begin{align}
\label{xs-def}
\|u\|_{X^{s}}
\doteq
 \sup_{t \in (0,T)} \|u\|_{\dot{H}^s} + \sup_{t \in (0,T)} t^\alpha \|u\|_{L^4}, \quad  
 \text{where} \quad \alpha = 
 \frac{1}{2} - \frac{s}{4}.
\end{align}
Finally, since we are working with two simultaneous equations, we define our norms on the product
spaces in the usual fashion.  For pairs, we have
\begin{align*}
\| (\varphi, \psi) \|_{\dot{H}^s \times \dot{H}^s}  &\doteq \|\varphi\|_{\dot{H}^s} +
\| \psi \|_{\dot{H}^s}, \\
\| (u,v) \|_{X^s \times X^s} &\doteq  \| u\|_{X^s}+ \|v\|_{X^s}.
\end{align*}

\textbf{The Diffusion Operator and Linear Estimates.}  
We begin by considering the diffusion equation 
\begin{equation} \label{heat}
\begin{split}
&w_t + w_{xx} = 0, \\
&w(0,x) = w_0(x).
\end{split}
\end{equation}
We take the solution operator $S$ of  \eqref{heat} as the Fourier multiplier
given by
\begin{align*}
\widehat{S(t) \varphi}(\xi) = e^{-t \xi^2} \widehat{\varphi}(\xi).
\end{align*}
Before we proceed with our estimates regarding the operator $S$, we  state the well-known 
Hardy-Littlewood-Sobolev inequality for the convenience of the reader.  For a  proof of this estimate, we refer
the reader to   \cite{lb}, \S 4.3.
\begin{lemma}[Hardy-Littlewood-Sobolev]
\label{hls}
 Suppose that  $0 < \alpha < 1$, $ 1 < \alpha < p < 1/\alpha $ and $q$ satisfies 
 \begin{align*}
     \frac{1}{q} 
     =
     \frac{1}{p} - \alpha.
 \end{align*}
If $f \in L^p$, then $D^{-\alpha}_x f \in L^q$ (where $D_x = (-\partial^{2}_x)^{1/2}$ is the Riesz potential) and there exists a constant $C= C (\alpha, p)$ such that
\begin{align}
 \| D^{-\alpha}_x f \|_{L^q}
\leq
    C \| f \|_{L^p}.
\end{align}
\end{lemma}

We now consider three estimates for $S$ that will be used throughout this work.  The first 
is an $L^p$-$L^q$ estimate that can be found  in \cite{gi}.
\begin{lemma}
\label{giga}
Let $S(t)$ be the solution operator for the heat equation \eqref{heat} with initial data $\varphi$  and $1 \le q \le p \le \infty$.  Then for $t > 0$ we have the estimate
\begin{align}
    \| S(t) \varphi \|_{L^p}
    \le
    (4 \pi t)^{-\frac{1}{2}(\frac{1}{q} - \frac{1}{p})}
    \|\varphi \|_{L^q},
\end{align}
\end{lemma}
%
%
%
\begin{lemma} 
\label{25}
For $s \ge 0$, there exists a positive constant $c_s$ such that
\begin{align*}
    \big\| D^{s} _x S(t) f \big\| _{L^2} \leq c_s t^{-s/2} \big\|f \big\|_{L^2}.
\end{align*}
\end{lemma}

\begin{proof}
By Parseval's identity and H\"{o}lder's inequality, we have
\begin{align}
\label{l-25-1}
     \big\| D^{s} _x S(t) f \big\| _{L^2} = 
     \frac{1}{2 \pi} \big\| |\xi|^{s} e^{-\xi^2t } \widehat{f} (\xi) \big\|_{L^2}
     \le      \frac{1}{2 \pi} \big\| |\xi|^{s} e^{-\xi^2t} \big\|_{L^{\infty}} \big\| \widehat{f} \big\|_{L^2}.
\end{align}
We see that the $L^\infty$-norm of $ |\xi|^{r} e^{-\xi^2t}$ is now easily bounded by 
\begin{align}
\label{parseval_2}
    \big\| |\xi|^{s} e^{-\xi^2t} \big\|_{L^{\infty}} \leq
    \Big(\frac{s}{2e^t} \Big)^{\frac{r}{2}} t^{-\frac{s}{2}}.
\end{align}
We can then continue bounding  \eqref{l-25-1} by using \eqref{parseval_2} to get
\begin{align*}
     \big\| D^{s} _x S(t) f \big\| _{L^2} 
     \leq
     \frac{1}{2 \pi} \Big( \frac{s}{2e^t} \Big)^{\frac{s}{2}} t^{-\frac{s}{2} } \big\| \widehat{f} \big\|_{L^2} 
     =
     \Big( \frac{s}{2e^t} \Big)^{\frac{s}{2}} t^{-\frac{s}{2}} \big\| f \big\|_{L^2}. 
\end{align*}
We see then that taking $c_s = ( \frac{s}{2} )^{\frac{s}{2}} $ we obtain our desired estimate.
\end{proof}
%
%
\begin{proposition}[Linear Estimate] 
\label{linear-est}
For $0 \le s  < 2$,
the mapping $\phi \mapsto S(t) \phi$, where $S$ is the solution operator to the diffusion equation,
continuously maps $\dot{H}^s \to X^s$, and we have the inequality
\begin{align}
\label{linear-ineq}
\|S(t) \varphi \|_{X^s} \le C_\ell  \| \varphi \|_{\dot{H}^s}.
\end{align}
where the constant $C_\ell$  depends on $s$ and $T$.
\end{proposition}
\begin{proof}
From the definition of the $X^s$-norm, we have
\begin{align*}
\| S(t) \varphi \|_{X^s}
=
\sup_{t \in (0,T)} \|S(t)\varphi \|_{\dot{H}^s} + \sup_{t \in (0,T)} t^\alpha \|S(t)\varphi\|_{L^4}.
\end{align*}
The first term is handled with the usual methods using Plancherel's Theorem and the definition 
of the diffusion solution operator $S$.  Indeed, we have
\begin{align*}
\|S(t) \varphi \|_{\dot{H}^s}^2
&=
\int_\rr
|\xi|^{2s} e^{-t \xi^2} | \widehat{\varphi} (\xi)|^2 d\xi
\le
\int_\rr|\xi|^{2s} | \widehat{\varphi} (\xi)|^2 d\xi
=
\|\varphi\|_{\dot{H}^s}^2.
\end{align*}
Thus, the first term has the upper bound 
\begin{align*}
\sup_{t \in (0,T)} \|S(t)\varphi\|_{\dot{H}^s} 
\le
\sup_{t \in (0,T)} \|\varphi\|_{\dot{H}^s}.
\end{align*}
For the second term, we begin by restricting our attention to the quantity under the $L^4$-norm.  
Applyling the Riesz Derivative and its inverse
with exponent $1/4$ gives us
\begin{align*}
\|S(t) \varphi \|_{L^4}  
=
\| D^{-1/4}_x D^{1/4}_x  S(t) \varphi   \| _{L^4} .
\end{align*}
Now we apply the Hardy-Littlewood-Sobolev Lemma, or  \Cref{hls}, with $p = 2,  q =4$ and $\alpha = 1/4$
to get
\begin{align*}
\| D^{-1/4}_x D^{1/4}_x  S(t) \varphi   \| _{L^4} 
\lesssim
\| D^{1/4}_x  S(t) \varphi \big \| _{L^2} .
\end{align*}
Next, we again apply the Riesz Derivative along with its inverse, both with exponent $s/2$,
after which
we can apply Lemma \ref{25} with the corresponding $\alpha = 1/4 - s/2$.  Note that this Lemma requires
$\alpha  \ge 0$, which corresponds to our hypothesis $0 \le s \le  1/2$.
We thus get
\begin{align*}
\| D^{1/4}_x  S(t) \varphi \big \| _{L^2}
=
\| D^{1/4 - s/2}_x   D^{s/2}_x S(t) \varphi \big \| _{L^2}
\lesssim
t^{-1/8 + s/4} \|  D^{s/2}_x\varphi \|_{L^2}
=
t^{ - 1/8 + s/4} \| \varphi \|_{\dot{H}^s}. 
\end{align*}
Using this upper bound for the $L^4$-norm of $S(t) \phi$, we get the time-weighted norm to be
\begin{align*}
\sup_{t \in (0,T)} t^\alpha \|S(t)\varphi\|_{L^4}
\lesssim 
\sup_{t \in (0,T)} t^\alpha  t^{-1/8 + s/4}   \|\varphi\|_{\dot{H}^s}.
\end{align*}
From the definition of the $X^s$ space, we have $\alpha = 1/2 - s/4$ which implies that the exponent for $t$ is $\alpha - \frac{1}{8} + \frac{s}{4} = \frac{3}{8}$.
Thus, we have the upper bound
\begin{align*}
\sup_{t \in (0,T)} t^{ \alpha - 1/8 + s/4}   \|\varphi\|_{\dot{H}^s}
\le T^{3/8} \| \varphi \|_{\dot{H}^s}.
\end{align*}
Putting our results together for each term gives the desired inequality \eqref{linear-ineq}, where the associated
constant with the upper bound depends on  $s$ and $T$.
\end{proof}

\section{Proof of Theorem 1}
\label{sec:1}
 \setcounter{equation}{0}

We begin by first reformulating the KM equations as a fixed point problem and then proving that the associated
integral operator is a contraction.  The Banach Contraction Mapping Theorem then implies the existence
and uniqueness of solutions to the KM ivp.  After this task is accomplished, we prove that the 
solution map $W$ is Lipschitz continuous.

{\bf Reformulating \eqref{km} as a fixed point problem.}  
We begin by taking the Fourier transform  in the spatial variable of both sides of the $u$ and
$v$ equations in \eqref{km}. 
Additionally, using the properties $\widehat{\p_x u} = i \xi \widehat{u}$ and
$\widehat{\p_t u} = \p_t \widehat{u}$, we get
\begin{align*}
& \partial_{t} \widehat{u} + \widehat{uv} + \xi^2  \widehat{ u} = 0, \\
& \partial_{t}\widehat{v} - \widehat{uv} + \xi^2 \widehat{ v} + \mu\widehat{v}= 0.
\end{align*}
Next, we isolate the terms corresponding to the linear heat equation on the left-hand side and multiply 
by the integrating factor $e^{t \xi^2}$.  The integrating factor allows us 
to then write the left-hand side of each equation as an exact derivative 
as for instance 
$\p_t (e^{t \xi^2} \widehat{u}) = e^{t \xi^2} ( \p_t \widehat{u} + \xi^2 u )$. 
We thus  obtain
\begin{align*}
&\p_t (e^{t \xi^2} \widehat{u}) =  - e^{t \xi^2}  \widehat{uv}, \\
&\p_t (e^{t \xi^2} \widehat{v}) =   e^{t \xi^2} \Big(  \widehat{uv} - \mu \widehat{v}\Big).
\end{align*}
We now integrate both equations from $0$ to $t$ using the Fundamental Theorem of Calculus and  the initial data specified in 
\eqref{km} to get
\begin{align*}
&e^{t \xi^2} \widehat{u} - \widehat{\varphi}  =  - \int_0^t  e^{t' \xi^2}  \widehat{uv} dt', \\
&e^{t \xi^2} \widehat{v} - \widehat{\psi} =  \int_0^t  e^{t' \xi^2} \Big(  \widehat{uv} - \mu \widehat{v}\Big) dt'.
\end{align*}
Next we move the initial data to the right-hand side of each equation and multiply both sides 
of both equations
by $e^{-t \xi^2}$.  Finally, applying the inverse Fourier Transform gives us 
\begin{align*}
&u(x,t) = \frac{1}{2\pi}\int_\mathbb{R} e^{ix\xi} e^{-t \xi^2}\widehat{\varphi} d\xi -
\frac{1}{2\pi}\int_\mathbb{R}  e^{ix\xi} e^{-t \xi^2} \Big(\int_{0}^{t} e^{t'\xi^2} \widehat{uv}dt'\Big) d\xi, \\
&v(x,t) = \frac{1}{2\pi}\int_\mathbb{R} e^{ix\xi} e^{-t \xi^2}\widehat{\psi} d\xi +
    \frac{1}{2\pi}\int_\mathbb{R}  e^{ix\xi} e^{-t\xi^2} \Big(\int_{0}^{t} e^{t'\xi^2} (\widehat{uv} -\mu\widehat{v}) dt'\Big).
\end{align*}
We see further that using the solution operator to the heat equation $S$, we can also rewrite
these equations as
%
\begin{align*}
&    u(x,t) = S(t)\varphi -  \int_{0}^{t} S(t - t') ( uv  )  dt', \\
&    v(x,t) = S(t)\psi + \int_{0}^{t} S(t -t') ( uv  -\mu v ) dt'.
\end{align*}
With the desired form of our equations obtained, we define the operators
\begin{align}
&T_1 (u,v) \doteq S(t)\varphi - \int_{0}^{t} S(t - t') ( uv  ) dt', \\
&T_2 (u,v) \doteq S(t)\psi + \int_{0}^{t} S(t -t') ( uv  -\mu v ) dt'.
\end{align}
We note here that the initial data $\varphi, \psi$ are assumed to be fixed in this construction.  Should the dependence on the initial data be needed, such as in the proof of Lipschitz continuity, we will denote
the operator with an appropriate subscript, for instance $T_1 = T_{1, \varphi}$. 
Finally, we take $T$ to be the two component operator 
\begin{align}
\label{T-def}
T(u,v) \doteq (T_1(u,v), T_2(u,v)).
\end{align}
We now can rewrite our equation as
\begin{align}
(u,v) = T(u,v).
\end{align}

Our next objective will be to demonstrate that $T$ has a fixed point using the Banach Fixed-Point Theorem.  One of the key ingredients we will need to establish a contraction is the following Bilinear Estimate, 
which is  proved  in the next section. 
\begin{proposition} [\textbf{Bilinear Estimate}]
\label{bilinear}
Let $0 \le s <  2$.  Then
there exists a constant $C_b>0$ such that for  $f,g \in X^s$ we have 
\begin{align}
\label{bilinear-ineq}
\Big\| \int_{0}^{t} S(t-t')(f g) dt' \Big\|_{X^s} \le
C_b \|f\|_{X^s}  \|g\|_{X^s}.
\end{align}
\end{proposition}
%


We now show the contraction, which completes the proof of well-posedness under a smallness asummption.

\begin{proposition}[Contraction]
\label{contraction}
Let $0 \le s < 2$,
$T < \frac{1}{6\mu}$,  $\rho = \frac{1}{3 C_b}$, 
$ (\varphi, \psi) \in \dot{H}^s \times \dot{H}^s$ satisfying the smallness condition
\begin{align}
\| (\varphi, \psi) \|_{\dot{H}^s \times \dot{H}^s} \leq 
\frac{1}{18 C_\ell C_b}, 
\end{align}
where the constants $C_\ell$ and $C_b$ are given in 
Propositions  \ref{linear-est}  and \ref{bilinear}, respectively.  
Then 
the operator $T$ defined in \eqref{T-def} maps $T: B(0, \rho) \to B(0,\rho)$, where 
$B(0,\rho)\dot{=} 
\{ (u,v) \in {X^s} \times {X^s}:\|(u,v) \|_{{X^s} \times {X^s}}
\leq 
\rho\}$
and 
is a contraction mapping.
\end{proposition}

The key to proving this proposition lies in the Bilinear Estimate, Proposition \ref{bilinear}, whose proof can be found in the next section.    We further note that the hypothesis on $s$ comes from the Bilinear Estimate.

\begin{proof}
We begin by first demonstrating that $T:B(0,\rho) \rightarrow B(0,\rho).$ 
Noting that 
\begin{align*}
\| T(u,v) \|_{X^s \times X^s}  =  \|T_1 (u,v) \|_{X^s} + \|T_2(u,v) \|_{X^s},
\end{align*}  
we proceed by examining the terms on the right-hand side separately.

{\bf Estimating $ \| T_1 (u,v) \|_{X^s} $.}  We have
\begin{align*}
\|T_1(u,v)\|_{X^s}
&=
\Big\|
S(t)\varphi - \int_{0}^{t} S(t - t') ( uv  ) dt'
\Big\|_{X^s}, \\
&\le
\Big\| S(t)\varphi  \Big\|_{X^s}
+
\Big\|  \int_{0}^{t} S(t - t') ( uv  ) dt' \Big\|_{X^s}.
\end{align*}
For the first term, we use Proposition \ref{linear-est}, which gives us
\begin{align*}
\| S(t)\varphi  \|_{X^s} \le C_\ell  \| \varphi \|_{\dot{H}^s}.
\end{align*}
The second term is bounded using Proposition \ref{bilinear}, giving us
\begin{align}
\label{ball-est-1}
\Big\|  \int_{0}^{t} S(t - t') ( uv  ) dt' \Big\|_{X^s}
\le
C_b \|u\|_{X^s} \|v\|_{X^s}
=
C_b \rho^2.
\end{align}
Therefore, we get
\begin{align}
\|T_1(u,v)\|_{X^s}
\le
C_\ell
\|\varphi\|_{\dot{H}^s}
+
C_b \rho^2.
\end{align}

{\bf Estimating $T_2 (u,v) $.}  We have
\begin{align*}
\|T_2(u,v)\|_{X^s}
&=
\Big\|
S(t)\psi +\int_{0}^{t} S(t -t') ( uv  -\mu v ) dt'
\Big\|_{X^s}, \\
&\le
\Big\| S(t)\psi  \Big\|_{X^s}
+
\Big\|  
\int_{0}^{t} S(t -t') ( uv ) dt'
\Big\|_{X^s}
+
\mu
\Big\|  
\int_{0}^{t} S(t -t') v  dt'
\Big\|_{X^s}.
\end{align*}
Again, the linear term is estimated using  Proposition \ref{linear-est}, giving us
\begin{align}
\label{linear T2}
\| S(t) \psi  \|_{X^s} \le C_\ell  \| \psi \|_{\dot{H}^s}.
\end{align}
The second term was already estimated in  \eqref{ball-est-1}.  Thus, we will 
restrict our attention to the third term, leaving out the coefficient of $\mu$ for the moment. 
We begin with the definition of the $X^s$-norm, noting that we will 
reinsert the constant $\mu$ at the end of the computation.  We have
\begin{align*}
&  \Big\| \int_{0}^{t} S(t-t') v dt' \Big\|_{X^s} 
=
 \sup_{t \in (0,T)}  \Big\| \int_{0}^{t} S(t-t') v dt' \Big\|_{\dot{H}^s} 
 + \sup_{t \in (0,T)}  t^\alpha  \Big\|   \int_{0}^{t} S(t-t') v dt'  \Big \|_{L^4}.
\end{align*}
Looking at the expression under the supremum, we see that for the Sobolev term
we have
\begin{align}
\notag
\Big\| \int_{0}^{t} S(t-t') v dt' \Big\|_{\dot{H}^s} 
&\le
\int_{0}^{t} 
\Big\| S(t-t') v dt' \Big\|_{\dot{H}^s},  \\\notag
&=
\int_{0}^{t} 
\Big(\int_\rr |\xi|^{2s} 
e^{-(t - t') \xi^2} | \widehat{v}(t')  |^2 d\xi \Big)^{1/2} dt',
\\
\notag
&\le
\int_{0}^{t} 
\Big(\int_\rr |\xi|^{2s}  | \widehat{v}(t') |^2 d\xi \Big)^{1/2} dt', \\
\label{mean-value-est}
&=
\int_{0}^{t} 
\|  v \|_{\dot{H}^s} dt'.
\end{align}
We can now apply the Mean-Value Theorem, which tells us that for some $t^\star \in (0,T)$, we have
\begin{align*}
\int_{0}^{t} 
\|  v\|_{\dot{H}^s}dt' = t^\star \| v(t^\star)  \|_{\dot{H}^s}.
\end{align*}
Taking the supremum over $t \in (0,T)$ thus gives us
\begin{align*}
\sup_{t \in (0,T)} \Big\| \int_{0}^{t} S(t-t') v dt' \Big\|_{\dot{H}^s}
&\le
\sup_{t \in (0,T)} 
t^\star\| v(t^\star) \|_{\dot{H}^s} 
\le
T \sup_{t \in (0,T)} \|  v \|_{\dot{H}^s}.
\end{align*}
For the time-weighted $L^4$-term, we first obtain an upper bound
by passing the norm into the integral.  We then apply the $L^p$--$L^q$ Linear Estimate, Lemma \ref{giga}, with
$p = q = 4$.  
\begin{align*}
\sup_{t \in (0,T)} t^\alpha  \Big\|   \int_{0}^{t} S(t-t')v dt'  \Big\|_{L^4}
\le
\sup_{t \in (0,T)} t^\alpha  \int_{0}^{t}   \|  v \|_{L^4}  dt'.
\end{align*}
Again, we  will apply the Mean-Value Theorem, giving us for some $t^\star \in (0,T)$, 
\begin{align*}
\sup_{t \in (0,T)} t^\alpha  \int_{0}^{t}   \| v \|_{L^4}  dt'
=
\sup_{t \in (0,T)} t^\alpha   t^\star  \|  v(t^\star)  \|_{L^4}  dt'.
\end{align*}
To bound this above and remove the dependence on $t^\star$, we note that 
$t^\star = (t^\star)^\alpha (t^\star)^{1- \alpha}$ and $0 \le \alpha \le 1$.  Then we use the fact that 
$t^\alpha \le T^\alpha$ and $(t^\star)^{1- \alpha} \le T^{1-\alpha}$.  Thus we get
\begin{align*}
\sup_{t \in (0,T)} t^\alpha   t^\star  \|  v(t^\star) \|_{L^4} 
&\le
T \sup_{t \in (0,T)}  (t^\star)^\alpha  \|  v(t^\star)  \|_{L^4}
\le
T \sup_{t \in (0,T)}  t^\alpha  \|  v  \|_{L^4}.
\end{align*}
Thus we can conclude
\begin{align}
\label{bilinear T2}
\mu
\Big\|   \int_{0}^{t}   v  dt' \Big\|_{X^s}
\le 
\mu T  \|v \|_{X^s} \le \mu T \rho.
\end{align}
Putting these results together, \eqref{linear T2} and \eqref{bilinear T2}, we get
\begin{align}
\label{ball-est-2}
\|T_2(u,v)\|_{X^s}
\le 
C_\ell
\|\psi \|_{\dot{H}^s} +
C_b \rho^2 +  \mu T \rho.
\end{align}

{\bf Estimating $ \| T (u,v) \|_{X^s \times X^s} $.}   Putting our results together, \eqref{ball-est-1} and \eqref{ball-est-2}, 
we get
\begin{align*}
\| T (u,v) \|_{X^s \times X^s} 
&=
\|T_1(u,v) \|_{X^s}
+
\|T_2(u,v) \|_{X^s}, \\
&\le
C_\ell \| (\varphi, \psi) \|_{\dot{H}^s \times \dot{H}^s}
+ 2 C_b \rho^2 + \mu T \rho,
\end{align*}
where $C_\ell$ is the constant associated with the upper bound from Lemma \ref{linear-est} and $C_b$
is the constant associated with the Bilinear Estimate, Proposition \ref{bilinear}.  Now using the hypothesis, 
$\rho = \frac{1}{3C_b}$ and $T < \frac{1}{6 \mu}$, we see that
\begin{align*}
C_\ell \| (\varphi, \psi) \|_{\dot{H}^s \times \dot{H}^s}
+ 2 C_b \rho^2 + \mu T \rho
&<
C_\ell \| (\varphi, \psi) \|_{\dot{H}^s \times \dot{H}^s} 
+ 2 C_b \Big( \frac{1}{3C_b} \Big) \rho + \mu \frac{1}{6\mu} \rho, \\
&=
C_\ell \| (\varphi, \psi) \|_{\dot{H}^s \times \dot{H}^s} 
+
\frac{5}{6} \rho.
\end{align*}
Thus, in order for $T$ to map $B(0,\rho)$  into $B(0,\rho)$, 
we must have
\begin{align*}
C_\ell \| (\varphi, \psi) \|_{\dot{H}^s \times \dot{H}^s} 
\le
\frac{1}{6} \rho,
\end{align*}
which is equivalent to our hypothesis
\begin{align*}
 \| (\varphi, \psi) \|_{\dot{H}^s \times \dot{H}^s} 
 \le \frac{1}{18 C_\ell C_b}.
\end{align*}

{\bf Contraction.}   
Recalling that the initial data $(\varphi, \psi)$ are fixed, we take two pairs of functions
$(u_1, v_1), (u_2, v_2) \in X^s \times X^s$ and compute the norm of 
$T(u_1, v_1) - T(u_2, v_2)$.  We break this  computation down by first examining the component
operators $T_1$ and $T_2$.

{\bf Estimating  the $T_1$ difference. }
We add and subtract a mixed term $u_1 v_2$  and regroup terms to obtain
\begin{align*}
T_1(u_1,v_1) -   T_1(u_2,v_2)
&=
-\int_{0}^{t} S(t-t') \Big( u_1v_1  - u_2 v_2\Big)dt', \\
&=
-\int_{0}^{t} S(t-t') \Big(  u_1 (v_1 - v_2) \Big)dt' 
-
\int_{0}^{t} S(t-t') \Big(  v_2 (u_1 - u_2)\Big)dt'.   
\end{align*}
We now examine this difference in the $X^s$ norm.  
After applying the triangle inequality, we use Bilinear Estimate, Proposition \ref{bilinear}, 
to further bound the nonlinearities.   We get
\begin{align*}
&\| T_1(u_1,v_1) -   T_1(u_2,v_2) \|_{X^s} \\
&\quad\quad\quad\le
\Big\|
\int_{0}^{t} S(t-t') \Big(  u_1 (v_1 - v_2) \Big)dt' 
\Big\|_{X^s}
+
\Big\|
\int_{0}^{t} S(t-t') \Big(  v_2 (u_1 - u_2)\Big)dt'
\Big\|_{X^s}, \\
&\quad\quad\quad\le
C_b 
\|u_1\|_{X^s} \|v_1 - v_2 \|_{X^s} 
+
C_b
\|v_2\|_{X^s}  \|u_1- u_2 \|_{X^s} .
\end{align*}
Now using the hypothesis that both $\|u_1\|_{X^s}$ and $\|v_2\|_{X^s}$ are bounded by 
the constant $\rho$,  and recalling that the norm on the product space is the sum of the norms, 
we get
\begin{align}
\label{T1-full-est}
\| T_1(u_1,v_1) -   T_1(u_2,v_2) \|_{X^s}
\le \rho C_b \| (u_1 - u_2, v_1 - v_2) \|_{X^s \times X^s}.
\end{align}

{\bf Estimating  the $T_2$ difference.} We first break up the integral to separate the linear term, giving us
\begin{align*}
&T_2(u_1,v_1) -     T_2(u_2,v_2) \\
&\quad\quad\quad\quad=
\int_{0}^{t} S(t-t') \Big( u_1v_1  - u_2 v_2\Big)dt
- \int_{0}^{t} S(t-t') \mu( v_1 -  v_2) dt'.
\end{align*}
The first term under the $X^s$ norm is precisely the same as that of $T_1$.  Thus, adding and subtracting
a mixed term  and then applying the Bilinear estimate leads to the same result.  
Thus we get
\begin{align}
\label{T2-diff}
\notag
&\| T_2(u_1,v_1) -     T_2(u_2,v_2) \|_{X^s} \\ 
&\quad\quad\quad\le  
\rho C_b \|(u_1  - u_2, v_1 - v_2)\|_{X^s \times X^s}
+
\mu \Big\| \int_{0}^{t} S(t-t') ( v_1 -  v_2) dt' \Big\|_{X^s}.
\end{align}
To estimate the linear term, we begin with the definition of the $X^s$-norm, noting that we will 
reinsert the constant $\mu$ at the end of the computation.  We have
\begin{align*}
&\Big\| \int_{0}^{t} S(t-t') ( v_1 -  v_2) dt' \Big\|_{X^s} \\
\notag
&\quad\quad\quad\ =
 \sup_{t \in (0,T)} \Big\| \int_{0}^{t} S(t-t') ( v_1 -  v_2) dt' \Big\|_{\dot{H}^s} 
 + \sup_{t \in (0,T)} t^\alpha \Big \|   \int_{0}^{t} S(t-t') ( v_1 -  v_2) dt'  \Big\|_{L^4}.
\end{align*}
The expression under the supremum can be bounded in precisely the same
fashion as \eqref{mean-value-est}, giving us 
\begin{align*}
\Big\| \int_{0}^{t} S(t-t') ( v_1 -  v_2) dt' \Big\|_{\dot{H}^s} 
&\le
\int_{0}^{t} 
\|  v_1 - v_2 \|_{\dot{H}^s} dt'.
\end{align*}
We can now apply the Mean-Value Theorem, which tells us that for some $t^\star \in (0,T)$, we have
\begin{align*}
\int_{0}^{t} 
\|  v_1 - v_2 \|_{\dot{H}^s}dt' = t^\star \| v_1(t^\star) - v_2(t^\star) \|_{\dot{H}^s}^2.
\end{align*}
Taking the supremum over $t \in (0,T)$ thus gives us
\begin{align}
\label{t2-sobolev}
\sup_{t \in (0,T)} \Big\| \int_{0}^{t} S(t-t') ( v_1 -  v_2) dt' \Big\|_{\dot{H}^s}
&\le  \notag
\sup_{t \in (0,T)} 
t^\star\| v_1(t^\star) - v_2(t^\star) \|_{\dot{H}^s}, \\ 
&\le
T \sup_{t \in (0,T)} \|  v_1 - v_2 \|_{\dot{H}^s}.
\end{align}
For the time-weighted $L^4$-term, we apply the $L^p$--$L^q$ Linear Estimate, Lemma \ref{giga}, with
$p = q = 4$.  
\begin{align*}
\sup_{t \in (0,T)} t^\alpha \Big \|   \int_{0}^{t} S(t-t') ( v_1 -  v_2) dt'  \Big\|_{L^4} 
&\le
\sup_{t \in (0,T)} t^\alpha  \int_{0}^{t}   \| S(t-t') ( v_1 -  v_2) \|_{L^4}  dt', \\
&\le
\sup_{t \in (0,T)} t^\alpha  \int_{0}^{t}   \|  v_1 -  v_2 \|_{L^4}  dt'.
\end{align*}
Again, we will apply the Mean-Value Theorem, giving us for some $t^\star \in (0,T)$, 
\begin{align*}
\sup_{t \in (0,T)} t^\alpha  \int_{0}^{t}   \|  v_1 -  v_2 \|_{L^4}  dt'
=
\sup_{t \in (0,T)} t^\alpha   t^\star  \|  v_1(t^\star) -  v_2(t^\star) \|_{L^4}  dt'.
\end{align*}
To bound this above and remove the dependence on $t^\star$, we note that 
$t^\star = (t^\star)^\alpha (t^\star)^{1- \alpha}$ and $0 \le \alpha \le 1$.  Then we use the fact that 
$t^\alpha \le T^\alpha$ and $(t^\star)^{1- \alpha} \le T^{1-\alpha}$.  Thus we get
\begin{align}
\label{t2-l4}
\sup_{t \in (0,T)} t^\alpha   t^\star  \|  v_1(t^\star) -  v_2(t^\star) \|_{L^4} 
&\le  \notag
T \sup_{t \in (0,T)}  (t^\star)^\alpha  \|  v_1(t^\star) -  v_2(t^\star) \|_{L^4}, \\
&\le
T \sup_{t \in (0,T)}  t^\alpha  \|  v_1 -  v_2 \|_{L^4}.
\end{align}
Putting everything together, \eqref{t2-sobolev} and \eqref{t2-l4}, we get
the estimate for the  term of $T_2$ as
\begin{align}
\notag
\mu  \Big\| \int_{0}^{t} S(t-t') ( v_1 -  v_2) dt' \Big\|_{X^s} 
&\le
\mu T \big( \sup_{t \in (0,T)} \|  v_1 - v_2 \|_{\dot{H}^s}
+
\sup_{t \in (0,T)}  t^\alpha  \|  v_1 -  v_2 \|_{L^4} \big), \\ 
\label{T2-diff-linear}
&\le
2 \mu T \| v_1 - v_2 \|_{X^s}.
\end{align}
Finally, we obtain the full estimate for $T_2$ as \eqref{T2-diff}, where the second term is bounded by
 \eqref{T2-diff-linear}. Hence
\begin{align}
\label{T2-ful-est}
&\| T_2(u_1,v_1) -     T_2(u_2,v_2) \|_{X^s} 
\le
( \rho C_b + 2 \mu T)  \|(u_1  - u_2, v_1 - v_2)\|_{X^s \times X^s}.
\end{align}

{\bf Combined Estimate for  the $T$ difference.}  Putting together our estimates for $T_1$ and $T_2$, namely, \eqref{T1-full-est} and \eqref{T2-ful-est},
we have
\begin{align}
\notag
&\|T(u_1, v_1) - T(u_2, v_2) \|_{X^s \times X^s} \\
\notag
&\quad\quad\quad=
\|T_1(u_1, v_1) - T_1(u_2, v_2) \|_{X^s }
+
\|T_2(u_1, v_1) - T_2(u_2, v_2) \|_{X^s }, \\
\label{contraction-est}
&\quad\quad\quad\le
(2 \rho C_b + 2 \mu T)  \|(u_1  - u_2, v_1 - v_2)\|_{X^s \times X^s}.
\end{align}
Our hypotheses that $\rho = \frac{1}{3C_b}$ and $T < \frac{1}{6\mu}$ now allow us to conclude 
that
\begin{align*}
2 \rho C_b + 2 \mu T < 1,
\end{align*}
and thus $T$ is indeed a contraction mapping.
\end{proof}

With the proof of Proposition  \ref{contraction} complete, we now are ready to prove 
Theorem \ref{main}. 
\begin{proof}[Proof of Theorem \ref{main}]
We see that the fixed point from Proposition \ref{contraction} gives us
the existence and uniqueness of the solutions to KM ivp \eqref{km}, thus we restrict our attention to proving that the
solution map $W$ is Lipschitz continuous.

Let $(\varphi_1, \psi_1), (\varphi_2, \psi_2) \in \dot{H}^s \times \dot{H}^s$, and
$(u_1, v_1), (u_2, v_2) \in X^s \times X^s$ be the corresponding solutions to the KM initial value 
problem with these initial data respectively.  Thus, for $W$ the solution operator to KM, 
our objective is to estimate the difference
\begin{align}
\label{W-dif}
\| W(\varphi_1, \psi_1) - W(\varphi_2, \psi_2) \|_{X^s \times X^s}
=
\|(u_1, v_1) - (u_2, v_2) \|_{X^s \times X^s}.
\end{align}

We see that we can reframe this question using the operator $T$ established above in 
\eqref{T-def}.  To use $T$, however, we must have fixed initial data.  Thus we will assume that
$T$ in this instance uses the initial data $(\varphi, \psi) = (0,0)$.  Thus, we can rewrite the difference
inside of the norm in \eqref{W-dif} as
\begin{align*}
W(\varphi_1, \psi_1) - W(\varphi_2, \psi_2)
=
(S(t) ( \varphi_1 - \varphi_2) , S(t) (\psi_1 - \psi_2)) +  (T(u_1,v_1) - T(u_2, v_2)).
\end{align*}
We therefore can now rewrite \eqref{W-dif} as
\begin{align*}
&\| W(\varphi_1, \psi_1) - W(\varphi_2, \psi_2) \|_{X^s \times X^s} \\
&\quad\quad\quad \le \notag
\| (S(t) ( \varphi_1 - \varphi_2) , S(t) (\psi_1 - \psi_2))  \|_{X^s \times X^s} 
+
\|(T(u_1,v_1) - T(u_2, v_2) \|_{X^s \times X^s}.
\end{align*}
For the first term on the right-hand side, we use our linear estimate from Lemma \ref{linear-est}, giving us 
\begin{align*}
\| (S(t) ( \varphi_1 - \varphi_2) , S(t) (\psi_1 - \psi_2))  \|_{X^s \times X^s} 
\le C_\ell \| ( \varphi_1 - \varphi_2 ,  \psi_1 - \psi_2)  \|_{\dot{H}^s \times \dot{H}^s}.
\end{align*}
For the second term, we will use the work from proving that $T$ is a contraction.   Setting
our contraction constant as 
$C = 2 \rho C_b + 2 \mu T < 1$, we use 
\eqref{contraction-est} to give us
\begin{align*}
\|(T(u_1,v_1) - T(u_2, v_2) \|_{X^s \times X^s} 
&\le 
C \|(u_1 - u_2 ,v_1 - v_2)  \|_{X^s \times X^s}, \\
&=
C \| W(\varphi_1, \psi_1) - W(\varphi_2, \psi_2) \|_{X^s \times X^s}.
\end{align*}
Thus, we get
\begin{align*}
&\| W(\varphi_1, \psi_1) - W(\varphi_2, \psi_2) \|_{X^s \times X^s}  \\ \notag
&\quad\quad\quad \le
C_\ell \| ( \varphi_1 - \varphi_2 ,  \psi_1 - \psi_2)  \|_{\dot{H}^s \times \dot{H}^s} 
+
C \| W(\varphi_1, \psi_1) - W(\varphi_2, \psi_2) \|_{X^s \times X^s}.
\end{align*}
Simplifying this inequality, gives us
\begin{align*}
&\| W(\varphi_1, \psi_1) - W(\varphi_2, \psi_2) \|_{X^s \times X^s}  
 \le
\frac{C_\ell}{1- C} \| ( \varphi_1 - \varphi_2 ,  \psi_1 - \psi_2)  \|_{\dot{H}^s \times \dot{H}^s},
\end{align*}
which implies that $W$ is Lipschitz continuous.
\end{proof}

\section{Bilinear Estimates}
\label{sec:1}
 \setcounter{equation}{0}

We begin with a proof of  Proposition \ref{bilinear}.    This proof in turn requires analogous
estimates in $\dot{H}^s$ and the time-weighted $L^4$-space, which were proved in Lemma
\ref{lem1}.

\begin{proof}[Proof of Proposition \ref{bilinear}]
Starting with  the definition of the $X^s$-norm, and passing the norms inside 
the integrals to bound above, we get 
\begin{align*}
&\Big\| \int_{0}^{t} S(t-t')(f g) dt' \Big\|_{X^s}  \\
&\quad\quad\quad\le
\sup_{t \in (0,T)}  \int_{0}^{t}  \Big\| S(t-t')(f g)\Big\|_{\dot{H}^s} dt' +
\sup_{t \in (0,T)} t^{\alpha}\int_{0}^{t}  \Big\| S(t-t')(f g)dt' \Big\|_{L^4} dt'.
\end{align*}
We now apply  \Cref{lem1} to each term on the right-hand side of this equation.
For the first term, we get 
\begin{align}
\label{bilinear-term-1}
& \sup_{t \in (0,T)}\int_{0}^{t} \Big\| 
S(t-t') (f g)(t') \Big\|_{\dot{H}^s} dt' \lesssim \sup_{t \in (0,T)} t^{2\alpha}  \|f  \|_{L^4} \|g  \|_{L^4}.
\end{align}
For the second term, we have
\begin{align}
\label{bilinear-term-2}
& \sup_{t \in (0,T)} t^{\alpha} \int_{0}^{t} \Big\|  S(t-t')(f g) \Big\|_{L^4} dt' 
\lesssim \sup_{t \in (0,T)} t^{2\alpha}  \|f  \|_{L^4} \|g  \|_{L^4}.
\end{align}
Combining \eqref{bilinear-term-1} and \eqref{bilinear-term-2}, we get
\begin{align*}
\Big\| \int_{0}^{t} S(t-t')(f g) dt' \Big\|_{X^s}  
&\lesssim
\sup_{t \in (0,T)}
t^{2\alpha} \| f\|_{L^4} \|g\|_{L^4},  \\
&\le
\Big(
\sup_{t \in (0,T)}
t^{\alpha} \| f\|_{L^4} 
\Big)
\Big(
\sup_{t \in (0,T)}
t^{\alpha} \| g\|_{L^4} 
\Big), \\
&\le
\|f\|_{X^s} \|g\|_{X^s}.
\end{align*}
The constant associated with the upper bound comes from Lemma \ref{lem1}, and as its particular
value is used to establish the Contraction in Theorem 1, we label it as  $C_b$.
\end{proof}

To complete the argument for Proposition \ref{bilinear}, we use the  following
beautiful calculus estimate related to the Beta distribution. A proof is given in \cite{ho}, A.2.
 
 \begin{lemma}
 \label{cal}
 Let $0 < a$, $b < 1$ 
 and $r = a + b -1 $ with
$0 \leq r \leq 1$,
then we have the following bound
\begin{align}
    B (a, b, t)
\dot{=}
    \int _{0}^{t} (t - y)^{-b}
    y^{-a} 
    dy
\leq
    c_{a, b} t^{-r}.
\end{align}
\end{lemma}
With this estimate in hand, 
we now proceed to compute the following estimates on the components of the $X^s$-norm, which completes
our arguments. 
\begin{lemma}
\label{lem1}

Let $f$ and $g \in X^{s}$, with $0  \le  s <  2$ and $0 \leq t < T < \infty$. Then the following inequalities hold 
\begin{align}
\label{bilinear-1}
    \sup_{t \in (0,T)}\int_{0}^{t}\Big\| S(t- t')(fg) \Big \|_{\dot{H}^s} dt' &\lesssim 
    \sup_{t \in (0,T)} t^{ 2\alpha }  \|f  \|_{L^4} \|g  \|_{L^4},  \\
    \label{bilinear-2}
    \sup_{t \in (0,T)} t^{ \alpha } \int_{0}^{t}\Big\| S(t- t')(fg) \Big \|_{L^4} dt' 
    &\lesssim \sup_{t \in (0,T)} t^{2\alpha }  \|f  \|_{L^4} \|g  \|_{L^4}.
\end{align}
\end{lemma}
\begin{proof} The argument for each of these estimates follows similar lines, but we examine 
each one separately.

\vspace{.1in}

{\bf Estimating \eqref{bilinear-1}.}
Using the definition of the Riesz Derivative, we begin with 
\begin{align*}
\int_{0}^{t}\Big\| S(t- t') (fg)\Big \|_{\dot{H}^s} dt'
=
\int_{0}^{t}\Big\|D_x^s S(t- t') (fg) \Big \|_{L^2} dt'
\end{align*}
We now apply Lemma \ref{25}, which requires $s \ge 0$,  to obtain
\begin{align*}
\int_{0}^{t}\Big\|D_x^s S(t- t') (fg) \Big \|_{L^2} dt' \leq  
\int_{0}^{t} (t- t')^{-\frac{s}{2}} \| f(t') g(t') \|_{L^2} dt'. 
\end{align*}
Next, the generalized Hölder inequality
allows us to break the $L^2$ norm in the integrand into a product of $L^4$ norms, 
giving us
\begin{align*}
\int_{0}^{t} (t- t')^{-\frac{s}{2}} \| f(t') g(t') \|_{L^2} dt'
\leq  
\int_{0}^{t} (t- t')^{-\frac{s}{2}} \|f(t') \|_{L^4}\|g(t') \|_{L^4}dt'.
\end{align*}
We now multiply and divide by $(t')^{2\alpha}$
and then pull out the factor  $(t')^{2\alpha} \| f(t')\|_{L^4} \|g(t')\|_{L^4}$
by taking a supremum over time.  This gives us the upper bound   
\begin{align}
\label{bilinear-inter-1}
\int_{0}^{t} (t- t')^{-\frac{s}{2}} \|f(t') \|_{L^4}\|g(t') \|_{L^4}dt' \leq 
\sup_{t \in (0,T)} t^{ 2 \alpha } \|f(t) \|_{L^4}\|g(t) \|_{L^4} \int_{0}^{t} (t- t')^{-\frac{s}{2}}  (t')^{- 2 \alpha} dt'.
\end{align}
 The remaining integral can now be handled by Lemma
 \ref{cal}. The choice of multiplying
 and dividing by $(t')^{2\alpha}$ is now apparent as this allows us to satisfy the hypothesis for 
 $r = \frac{s}{2}  + 2\alpha - 1$.  Noting that we require $0 \le r \le 1$, 
 we see that our definition of $\alpha = \frac{1}{2} - \frac{s}{4}$ gives us 
\begin{align*}
r 
=  \frac{s}{2}  + 2\alpha - 1  = \frac{s}{2} + 2 \Big(\frac{1}{2} - \frac{s}{4} \Big) - 1 = 0.
\end{align*}
We see here that this $r$ restriction is always satisfied regardless of the value of $s$.  
The additional requirements of Lemma \ref{cal} are satisfied so long as we take
$0 \le s < 2$ and $0 < \alpha \le 1/2$.
Thus for a constant that only depends on $s$, as we have   $\alpha$   a function of $s$, 
we get
\begin{align*}
\int_{0}^{t} (t- t')^{-\frac{s}{2}}  (t')^{- 2 \alpha} dt'
< c_{s}.
\end{align*}
We therefore further bound \eqref{bilinear-inter-1} by 
\begin{align*}
\sup_{t \in (0,T)} t^{ 2 \alpha } \|f(t) \|_{L^4}\|g(t) \|_{L^4} \int_{0}^{t} (t- t')^{-\frac{s}{2}}  (t')^{- 2 \alpha} dt'
\le
c_s
\sup_{t \in (0,T)} t^{ 2 \alpha } \|f(t) \|_{L^4}\|g(t) \|_{L^4}.
\end{align*}
This chain of inequalities thus establishes \eqref{bilinear-1}.

{\bf Estimating \eqref{bilinear-2}.}
Before estimating the full expression on the left-hand side of \eqref{bilinear-2} 
we first estimate its integrand. 
To accomplish this, 
we begin by using Lemma \ref{giga} which is the $L^p$-$L^q$ Heat Kernel Estimate
and the   generalized Hölder inequality 
to get
\begin{align*}
\| S(t- t') (fg) \|_{L^4}
&\lesssim
 (t - t')^{-\frac{1}{8}}
 \| f  (t') g (t')  \|_{L^2}
 \lesssim
 (t - t')^{-\frac{1}{8}}
 \| f  (t') \|_{L^4} \| g (t')  \|_{L^4}.
\end{align*}
This upper bound for $\| S(t- t') (fg) \|_{L^4}$ thus
gives us
\begin{align}
\label{alpha-ub-1}
\sup_{t \in (0,T)} t^{ \alpha } \int_{0}^{t}\Big\| S(t- t')(fg) \Big \|_{L^4} dt' 
\lesssim
\sup_{t \in (0,T)} 
 t^{\alpha }  \int_{0}^{t}
(t - t')^{-\frac{1}{8}}
\| f   (t') \|_{L^4} \| g (t')  \|_{L^4} dt'.
\end{align}
While our construction  would suggest examining 
$0 < \alpha \le \frac{1}{2}$, we in fact are able to prove this estimate for $0 < \alpha \le \frac{7}{8}$; though
the values of $\alpha > \frac{1}{2}$ are unused.
To continue estimating \eqref{alpha-ub-1}, our argument  splits based into the cases where $0 < \alpha \le \frac{7}{16}$ and
$\frac{7}{16} < \alpha \le \frac{7}{8}$.   The strategy in both cases is similar, with the difference lying in 
the quantity we multiply and divide by in order to utilize Lemma \ref{cal}.

{\bf The case $0 < \alpha \le \frac{7}{16}$.}
In this case, we   multiply and divide by $(t')^{\frac{7}{8}}$ inside the 
integrand
and then pull out the factor of
$(t') ^{\frac{7}{8}} \|f(t') \|_{L^4}\|g(t') \|_{L^4}$ by taking a supremum over time.
We thus get
\begin{align}
\notag
&\sup_{t \in (0,T)} 
 t^{\alpha }  \int_{0}^{t}
(t - t')^{-\frac{1}{8}}
\| f   (t') \|_{L^4} \| g (t')  \|_{L^4} dt' \\
\label{bilinear-inter-2}
&\quad\quad\quad\quad
\le
\sup_{t \in (0,T)} 
 t^{\alpha}
\Big[
\sup_{t' \in (0, t)}
(t') ^{\frac{7}{8}} \|f(t') \|_{L^4}\|g(t') \|_{L^4}
\Big]
\int_{0}^{t} 
(t- t')^{-\frac{1}{8}} (t') ^{- \frac{7}{8} } dt'.
\end{align}
The integral in \eqref{bilinear-inter-2}, can now be handled with
the Beta distribution estimate,  Lemma \ref{cal}.  We see  that
the corresponding $r$ will be $r = 1/8 + 7/8 - 1= 0$, and therefore
\begin{align*}
\int_{0}^{t} 
(t- t')^{-\frac{1}{8}} (t') ^{- \frac{7}{8} } dt'
\lesssim 1.
\end{align*}
To handle the composition of suprema, we first note that 
as $\alpha > 0$, we have $t^\alpha \le T^\alpha$.  Then the interior supremum 
can be bounded above by taking the supremum to be over the full interval $(0,T)$.
Putting these estimates together gives us
\begin{align*}
&\sup_{t \in (0,T)} 
 t^{\alpha}
\Big[
\sup_{t' \in (0, t)}
(t') ^{\frac{7}{8}} \|f(t') \|_{L^4}\|g(t') \|_{L^4}
\Big]
\int_{0}^{t} 
(t- t')^{-\frac{1}{8}} (t') ^{- \frac{7}{8} } dt'  
\\
&\quad\quad\quad\quad\quad\lesssim
T^\alpha 
\sup_{t \in (0,T)} 
t^{\frac{7}{8}} \|f(t) \|_{L^4}\|g(t') \|_{L^4}
.
\end{align*}
Thus, so long as $ 0 < 2 \alpha \le 7/8$ we see that
$t^{\frac{7}{8} - 2\alpha}$ will be positive, and we thus continue our estimation by multiplying and dividing  by 
$t^{2\alpha}$.  After bounding above by pulling out the factor of $t^{\frac{7}{8} - 2\alpha}$,
we get 

\begin{align*}
T^\alpha 
\sup_{t \in (0,T)} 
t^{\frac{7}{8}} \|f(t) \|_{L^4}\|g(t') \|_{L^4}
\le T^{\frac{7}{8} - \alpha}
\sup_{t \in (0,T)} 
t^{2 \alpha} \|f(t) \|_{L^4}\|g(t') \|_{L^4}.
\end{align*}
Chaining these inequalities gives us \eqref{bilinear-2}.

{\bf The case $\frac{7}{16} < \alpha \le \frac{7}{8}$.}  To continue bounding \eqref{alpha-ub-1}, 
we follow a similar approach to the above case but 
alternatively multiply and divide by $(t')^{2 \alpha}$.  We thus obtain
\begin{align}
\notag
&\sup_{t \in (0,T)} 
 t^{\alpha }  \int_{0}^{t}
(t - t')^{-\frac{1}{8}}
\| f   (t') \|_{L^4} \| g (t')  \|_{L^4} dt' \\
\label{bilinear-inter-3}
&\quad\quad\quad\quad
\le
\sup_{t \in (0,T)} 
 t^{\alpha}
\Big[
\sup_{t' \in (0, t)}
(t') ^{2\alpha} \|f(t') \|_{L^4}\|g(t') \|_{L^4}
\Big]
\int_{0}^{t} 
(t- t')^{-\frac{1}{8}} (t') ^{- 2\alpha} dt'.
\end{align}
Our hypothesis of $ \frac{7}{16} < \alpha \le \frac{7}{8}$ 
in this case satisfies the requirement of Lemma \ref{cal}.
We see that for 
$r =  \frac{1}{8} + 2\alpha - 1$ to satisfy
$0 \le r \le 1$ is equivalent to 
$\frac{7}{16} \le \alpha \le \frac{15}{16}$. 
Thus the Beta distribution estimate implies that we have
\begin{align*}
\int_{0}^{t} 
(t- t')^{-\frac{1}{8}} (t') ^{- 2\alpha} dt'
\lesssim
t^{-\frac{1}{8} - 2 \alpha + 1}.
\end{align*}
Continuing our estimation on \eqref{bilinear-inter-3}, and noting that 
$\alpha  - r = \frac{7}{8} - \alpha$, we get
\begin{align}
\notag
&\sup_{t \in (0,T)} 
 t^{\alpha}
\Big[
\sup_{t' \in (0, t)}
(t') ^{2\alpha} \|f(t') \|_{L^4}\|g(t') \|_{L^4}
\Big]
\int_{0}^{t} 
(t- t')^{-\frac{1}{8}} (t') ^{- 2\alpha} dt' \\
&\quad\quad\quad\quad
\lesssim
\sup_{t \in (0,T)} 
 t^{\frac{7}{8} - \alpha}
\Big[
\sup_{t' \in (0, t)}
(t') ^{2\alpha} \|f(t') \|_{L^4}\|g(t') \|_{L^4}
\Big]
\end{align}
To bound the exterior supremum,  we see that our hypothesis in this case
implies  that $\frac{7}{8} - \alpha \ge 0$.  We  therefore get
\begin{align*}
\sup_{t \in (0,T)} 
 t^{\frac{7}{8} - {\alpha}}
\Big[
\sup_{t' \in (0, t)}
(t') ^{2\alpha} \|f(t') \|_{L^4}\|g(t') \|_{L^4}
\Big]
\le
T^{\frac{7}{8} - {\alpha}}
\sup_{t \in (0, T)}
t^{2\alpha} \|f(t) \|_{L^4}\|g(t) \|_{L^4}.
\end{align*}
Chaining these inequalities thus gives us \eqref{bilinear-2}.
\end{proof}

%
%
%
%

%

\vskip0.2in
\begin{minipage}[b]{9 cm}
   Curtis Holliman ({\it Corresponding author})  \\
    Department of Mathematics \\ 
     The Catholic University of America\\
Washington, DC 20064\\
      E-mail: {\it holliman$@$cua.edu}
\end{minipage}
\hfill
\begin{minipage}[b]{7 cm}
\noindent{Harry Prieto}\\
Department of Mathematics\\
The Catholic University of America\\
Washington, DC 20064\\
E-mail: {\it 22prieto@cua.edu}
\end{minipage}

\end{document}